\documentclass[11pt]{amsart}
\usepackage{amssymb,amscd,amsthm,latexsym}
\allowdisplaybreaks[3]
\usepackage{graphicx}
\usepackage[british]{babel}
\usepackage[all]{xy}
\usepackage{color}
\usepackage{mathabx}
\usepackage{calrsfs}
\usepackage{booktabs}

\theoremstyle{plain}
\newtheorem{theorem}[subsection]{Theorem}

\newtheorem{proposition}[subsection]{Proposition}
\newtheorem{corollary}[subsection]{Corollary}

\theoremstyle{definition}
\newtheorem{definition}[subsection]{Definition}

\theoremstyle{remark}
\newtheorem{example}[subsection]{Example}

\newtheorem{remark}[subsection]{Remark}

\newdir{ >}{
 @{}*!/-10pt/@{>} }
\newdir{ |>}{
 @{}*!/-5.5pt/@{|}*!/-10pt/:(1,-.2)@^{>}*!/-10pt/:(1,+.2)@_{>} }
\newdir{ >>}{{}*!/8pt/@{|}*!/3.5pt/:(1,-.2)@^{>}*!/3.5pt/:(1,+.2)@_{>}}

\newcommand{\defn}[1]{\emph{#1}}

\newcommand{\To}{\Rightarrow}

\def\mathrmdef#1{\expandafter\def\csname#1\endcsname{{\rm#1}}}
\mathrmdef{max}
\mathrmdef{op}
\mathrmdef{id}

\newcommand{\C}{\ensuremath{\mathbb{C}}}

\newcommand{\V}{\ensuremath{\mathcal{V}}}
\newcommand{\Grp}{\ensuremath{\mathsf{Grp}}}
\newcommand{\Ord}{\ensuremath{\mathsf{Ord}}}
\newcommand{\Pos}{\ensuremath{\mathsf{Pos}}}
\newcommand{\Ab}{\ensuremath{\mathsf{Ab}}}
\newcommand{\Mon}{\ensuremath{\mathsf{Mon}}}

\newcommand{\N}{\ensuremath{\mathbb{N}}}

\newcommand{\Pt}{\ensuremath{\mathsf{Pt}}}

\newcommand{\twosum}[2]{\protect{\left(\begin{smallmatrix} {#1}\\ {#2}\\ \end{smallmatrix}\right)}}

\def\mathrmdef#1{\expandafter\def\csname#1\endcsname{{\rm#1}}}
\newcommand{\la}{\langle}
\newcommand{\ra}{\rangle}
\newcommand{\mono}{\rightarrowtail}

\newcommand{\moi}{\preccurlyeq}

\def\vpullback{
\ar@{-}[]+D+<6pt,-6pt>;[]+D+<0pt,-12pt>;%
\ar@{-}[]+D+<0pt,-12pt>;[]+D+<-6pt,-6pt>}

\def\pullback{
 \ar@{-}[]+R+<4pt,-3pt>;[]+RD+<4pt,-6pt>%
 \ar@{-}[]+D+<3pt,-6pt>;[]+RD+<4pt,-6pt>}

\usepackage{chngcntr}
\counterwithin*{equation}{section}

\def\mathrmdef#1{\expandafter\def\csname#1\endcsname{{\rm#1}}}
\mathrmdef{max}
\mathrmdef{op}
\mathrmdef{Nat}
\mathrmdef{sym}

\def\xx{\mathsf{x}}
\def\yy{\mathsf{y}}

\newcommand{\comma}[2]{ #1 \! \downarrow\! #2}
\newcommand{\J}{\ensuremath{\mathcal{J}}}
\newcommand{\T}{\ensuremath{\mathcal{T}}}
\newcommand{\Alg}[1]{\ensuremath{\mathsf{Alg}(#1)}}

\begin{document}

\title{A note on varieties of ordered algebras}

\author{Maria Manuel Clementino}
\address{CMUC, Department of Mathematics, University of
Coimbra, 3000-143 Coimbra, Portugal}\thanks{}
\email{mmc@mat.uc.pt}

\author{Diana Rodelo}
\address{Department of Mathematics, University of the Algarve, 8005-139 Faro, Portugal and Center for Research and Development in Mathematics and Applications (CIDMA), Department of Mathematics, University of Aveiro, 3810-193 Aveiro, Portugal}
\thanks{The authors acknowledge partial financial support by the Fundação para a Ciência e a Tecnologia (Portuguese Foundation for Science and Technology) under the scope of the projects UID/00324/2025 (https://doi.org/10.54499/UID/00324/2025) (Centre for Mathematics of the University of Coimbra).
The second author also acknowledges financial support by CIDMA (https://ror.org/05pm2mw36)
under the Portuguese Foundation for Science and Technology
(FCT, https://ror.org/00snfqn58), Grants UID/04106/2025 (https://doi.org/10.54499/UID/04106/2025)
and UID/PRR/04106/2025.
}
\email{drodelo@ualg.pt}



\begin{abstract}
The aim of this work is to study the notions of lax protomodular and $\Ord$-Mal'tsev category
at the level of (coherent) varieties of (pre)ordered algebras and to further compare them, as has been done in the non-ordered context. We characterise varieties of ordered algebras which are (co)lax protomodular and those which are $\Ord$-Mal'tsev, in terms of operations of arities given by ordered sets and inequalities involving them.
We exhibit examples of (co)lax protomodular non-degenerate $\Ord$-categories, which were unknown. We prove that, for varieties of ordered algebras which are $\Ord$-Mal'tsev categories, the order of their algebras is degenerate (i.e. is symmetric). As a consequence, the implication ``protomodular $\Rightarrow$ Mal'tsev'' cannot be carried out to our 
context. The case of non-coherent ordered varieties which are $\Ord$-Mal'tsev categories is also addressed, where we show the existence of algebras with non-degenerate order.
\end{abstract}

\maketitle

\section*{Introduction}
The notions of protomodular category~\cite{NEKEAC} and Mal'tsev category~\cite{CLP} (see also~\cite{CKP, CPP, BB}) have played a crucial role in the development of categorical algebra over the past decades. Both give good generalisations of the variety $\Grp$ of groups, thus many features of groups have been extended to these categorical contexts (see~\cite{BB}, for example). The notion of protomodular category is stronger than that of a Mal'tsev category, therefore offering a wider range of tools to work with. For instance, pointed protomodular categories satisfy the Split Short Five Lemma and give a good context to develop homological algebra.

In~\cite{NEKEAC} the notion of \emph{protomodular category} was introduced as a finitely complete\footnotemark\footnotetext{The original definition only asks for the existence of pullbacks of points along arbitrary morphisms.} category $\C$ such that the pullback change-of-base functors $g^*\colon \Pt_B(\C) \to \Pt_C(\C)$ are conservative for any morphism $g\colon C\to B$. If $\C$ is pointed, then $\C$ is protomodular if and only if the Split Short Five Lemma holds. The leading example of a protomodular category is $\Grp$. More precisely, in \cite{BJ} the varieties of universal algebras which form a pointed protomodular category were characterised as those whose theory admits a unique constant $0$, and, for some $n\in\mathbb{N}$, $n$ binary operations $\alpha_1,\dots,\alpha_n$ and an ($n+1$)-ary operation $\theta$ such that
\begin{equation}\label{proto ops}
	\left\{\begin{array}{l} \alpha_i(x,x)=0, \text{ for all }i\in\{1,\dots,n\} \\ \theta(\alpha_1(x,y),\dots,\alpha_n(x,y),y)=x.
	\end{array}\right.
\end{equation}
For $\Grp$, $n=1$, $\alpha(x,y)=xy^{-1}$ and $\theta(x,y)=xy$.

In~\cite{CMR} we presented a lax version of the algebraic property of protomodularity in the context of $\Ord$-enriched categories, where $\Ord$ is the category of (pre)ordered sets and monotone maps. In that work we considered the main characteristic properties of protomodularity, with respect to comma objects instead of pullbacks, and showed that those characteristic properties also hold in the enriched context. We called these \emph{(co)lax protomodular $\Ord$-enriched categories}. Since taking comma objects concerns an ordered pair of morphisms, this construction leads to two change-of-base functors; we call them \emph{vertical} and \emph{horizontal} change-of-base functors. The conservativeness of the vertical one defines lax protomodularity and the conservativeness of the horizontal one defines colax protomodularity (Section~\ref{Section: proto}).

One of the aims of this work is the characterisation of pointed (co)lax protomodular varieties of ordered algebras \eqref{inequalities for lax proto} of Theorem~\ref{lax proto ord-varieties}, similar to the one in \eqref{proto ops} for the non-enriched varietal context.
Due to this result we were able to produce examples of (co)lax protomodular non-degenerate $\Ord$-categories and give a positive answer to the question (raised in~\cite{CMR}) of their existence.

A variety of universal algebras is called a \emph{$2$-permutable variety}~\cite{Smith} (also called \emph{congruence permutable} or \emph{Mal'tsev} variety) when its theory admits a ternary Mal'tsev operation $p$ such that
\begin{equation}\label{Maltsev op}
	\left\{\begin{array}{l} p(x,y,y)=x \\ p(x,x,y)=y .
	\end{array}\right.
\end{equation}
It is well-known that $\Grp$ is $2$-permutable, with $p(x,y,z)=xy^{-1}z$. The existence of a Mal'tsev operation is equivalent to the fact that the composition of any pair of congruences $R,S$ on an algebra $A$ of $\V$ is congruence permutable: $RS\cong SR$~\cite{Maltsev-Sbornik}. Since this property makes sense in any regular category, it was possible to extend this notion to a categorical context in~\cite{CLP}: a regular category $\C$ is called a \emph{Mal'tsev category} when any pair of equivalence relations $R, S$  on any object $A$ in $\C$ is such that $RS\cong SR$. In~\cite{CR} we introduced the notion of an \emph{$\Ord$-Mal'tsev category} and studied several characterisations. We start Section~\ref{Section: Mal'tsev} by recalling the basic aspects of this work.

In Theorem~\ref{Ord-Maltsev ord-varieties} we obtain a characterisation of the varieties of ordered algebras which are $\Ord$-Mal'tsev categories. The theory admits a ternary operation satisfying conditions similar to those in \eqref{Maltsev op} for the non-enriched varietal context, where, as expected, inequalities play an essential role. Based on this characterisation, we prove that all algebras in a variety of ordered algebras which is an $\Ord$-Mal'tsev category have degenerate order, i.e. the order of the algebras is symmetric, thus an equivalence relation. This leads to an interesting remark:
lax protomodular varieties of ordered algebras do not need to be $\Ord$-Mal'tsev.

In Section~\ref{Section: non-coherent} we consider the context of non-coherent varieties of ordered algebras. We find an example of an $\Ord$-Mal'tsev non-coherent ordered variety whose algebras have non-degenerate order.

\section{Varieties of ordered algebras}\label{Vars of ordered algs}

We shall consider (finitary) varieties of ordered algebras, where ``order'' here means preorder (i.e. a reflexive and transitive relation). Although most of the known literature on ordered algebras is carried out for the partial order case (see~\cite{Adamek-Milius, KV}), we prefer to work in the more general setting of preorders since it has proved to be more interesting from a categorical
point of view, as illustrated by recent results in preordered groups~\cite{CM-FM}.

In this section we present the setting to define varieties of ordered algebras. The approach is very similar to the usual one for (discrete) varieties of algebras (from~\cite{Lawvere}), where the main difference is in the system of arities: an arity is a finite ordered set $(X,\le)$ instead of a natural number $n$ (which we will denote simply by $X$). We follow closely~\cite{Power, PPower, NPower}
and the more recent approach of~\cite{LR}.

First we consider the \emph{system $\J$ of arities}, which is the full subcategory of $\Ord$ of finite ordered sets.
 An \emph{$\Ord$-algebraic theory with arities $\J$} is an $\Ord$-category $\T$ (i.e. a category enriched in the category $\Ord$) with finite powers together with an $\Ord$-functor $\tau\colon \J^\op \to \T$ which preserves finite powers and is surjective on objects. A \emph{$\T$-algebra} is a finite power preserving $\Ord$-functor $A\colon \T \to \Ord$, while a $\T$-algebra homomorphism is a natural transformation $f:A\To B$ between $\T$-algebras $A$ and $B$. If $f$ is component-wise monomorphic, we say that $A$ is a \emph{$\T$-subalgebra} of $B$.

 Recall that in $\Ord$, and therefore also in $\J$, the \emph{power of ordered sets $C$ and $ X$} (also called the \emph{cotensor of $C$ by $X$}) is the ordered set \[X\pitchfork C=C^X=\{\xx\colon X\to C: \xx\text{ is monotone}\},\]
with its pointwise order. Moreover, in $\Ord$, and therefore also in $\J$, every (finite) ordered set $X$ is the \emph{copower} (or \emph{tensor}) of $1$ by $X$, $1\bullet X$.
Hence, if we denote $\tau(1)= T$, we
  have
  \[\tau(X)=\tau(X\pitchfork 1)=X\pitchfork T,\]
  and a $\T$-algebra $A\colon\T\to\Ord$ is completely determined by the image $A(T)$ of $T$, since, for any $X$ we have
  \[A(X\pitchfork T)=A(T)^X,\]
  with its pointwise order: given monotone maps $\xx,\yy\colon X \to A(T)$, we define $\xx\le \yy$ if and only if $\xx(x)\le \yy(x)$ for all $x\in X$. Moreover, a $\T$-algebra homomorphism is completely determined by its $T$-component. As a consequence, it is common to ``identify''  a $\T$-algebra $A$ with the ordered set  $(A(T), \le)$ and a $\T$-algebra homomorphism with $f_T\colon A(T)\to B(T)$, and simplify the notation to $(A,\le)$, or simply $A$, and $f\colon A\to B$.

  An \emph{operation with arity $(X,\le)$} is a morphism $\theta\colon X\pitchfork T  \to T$ in $\T$. The image $A(\theta)$ of an operation $\theta\colon X\pitchfork T\to T$ of arity $X$ is a monotone map $\theta_A\colon A^X\to A$ (where, when convenient, we omit the subscript). When $X$ has $n$ elements and discrete order, we simply say that $\theta$ is an $n$-ary operation.
  Since a $\T$-algebra homomorphism $f\colon A\to B$ is a natural transformation, for each operation $\theta\colon X\pitchfork T\to T$ the following diagram commutes
  \[\xymatrix{A^X\ar[r]^{f^X}\ar[d]_{\theta_A}&B^X\ar[d]^{\theta_B}\\
  A\ar[r]_f&B}\]
  and, consequently, for any $\xx\colon X\to A$ in $A^X$, $f(\theta_A(\xx))=\theta_B(f\xx).$

  We point out that the notion of operation also captures that of \emph{partial operation} in the following sense: given $\theta\colon X\pitchfork T\to T$ in $\T$, if, for instance, $X=\{x_1,x_2,x_3\}$ is endowed with the discrete order, then $A^X\cong A^3$ and so $\theta_A$ is defined for all triples $(a_1,a_2,a_3)$ of elements of $A$, while if $X=\{x_1,x_2,x_3\}$ with $x_2\leq x_3$ then $A^X\cong\{\xx\colon X\to A: \xx$ is monotone$\}\cong\{(a_1,a_2,a_3)\in A^3\,;\,a_2\leq a_3\}$, and so $\theta_A$ is defined for $(a_1,a_2,a_3)$ only if $a_2\leq a_3$. Moreover, the $\Ord$-enrichment of $\T$ induces an order between operations of the same arity $X$, namely the order of $\T(X\pitchfork T,T)$. In particular, if $x\leq y$ in $X$, then $\pi_x=\tau(i_x)\leq\tau(i_y)=\pi_y$, where $i_x\colon 1\to X$ is the map that selects $x$.

The category $\Alg{\T}$ of $\T$-algebras and $\T$-algebra homomorphisms is $\Ord$-enriched by the pointwise preorder. We say that an $\Ord$-enriched category $\V$ is a \emph{variety of ordered algebras} when $\V\cong \Alg{\T}$, for some $\Ord$-algebraic theory $\T$. A variety $\V$ is pointed when there exists precisely one operation of arity 0 in $\T$, i.e. when $\T$ admits one constant, usually denoted by 0.

The category $\Alg{\T}$ admits all weighted (co)limits, and its forgetful functor $U\colon\Alg\T\to\Ord$ has an $\Ord$-enriched left adjoint $F$, defined by
$F(1)=\T(T,-)$,
and, consequently, for every $X\in\J$, $F(X)=\T(X\pitchfork T,-)$ because, for any $\T$-algebra $A$,
\[\begin{array}{rcll}
\Alg{\T}(\T(X\pitchfork T,-),A)&\cong&\Nat(\T(X\pitchfork T,-),A)\\
&\cong&A(X\pitchfork T)&\mbox{ (by Yoneda Lemma)}\\
&\cong&X\pitchfork A(T)&\mbox{ ($A$ preserves powers)}\\
&\cong&\Ord(X,A(T))\cong \Ord(X,U(A)).\end{array}\]

\section{Protomodularity of varieties of ordered algebras}\label{Section: proto}

We collect from~\cite{CMR} the main aspects concerning an order enrichment of the notion of protomodular category~\cite{NEKEAC}. The main idea is to consider some of the characteristic properties of protomodularity with respect to
comma objects instead of pullbacks.

Let $\C$ denote a finitely complete $\Ord$-category with comma objects. Given an object $B$, as usual we denote by $\C/B$ the slice category of $\C$ over $B$. A \defn{point over $B$} is a morphism from the terminal object $1_B\colon B\to B$ into an arbitrary object $f\colon A\to B$ of $\C/B$; that is, a $\C$-morphism $s\colon B\to A$ so that $fs=1_B$. Hence a point in $\C$ is given by a split epimorphism $f\colon A\to B$ with a chosen splitting $s\colon B\to A$. We denote by $\Pt_B(\C)$ the category of points over $B$ in $\C$.

Given a morphism $g\colon C\to B$ in $\C$, we can define two possible functors by taking comma objects along points. For a point $\xymatrix{A \ar@<2pt>[r]^-f & B \ar@<2pt>[l]^-s}$, we form the comma objects
\[
\xymatrix@=30pt{\comma{g}{f} \ar[r]^-{\pi_2} \ar@<-2pt>[d]_-{\pi_1} \ar@{}[dr]|-{\le} & A \ar@<-2pt>[d]_-f \\ C \ar@<-2pt>@{-->}[u]_-{s_1}\ar[r]_-g & B \ar@<-2pt>[u]_-s}
\;\;\;\;\;\;\text{and}\;\;\;\;\;\;
\xymatrix@=30pt{\comma{f}{g} \ar@<2pt>[r]^-{\rho_2} \ar[d]_-{\rho_1} \ar@{}[dr]|-{\le} & C \ar[d]^-g \ar@<2pt>@{-->}[l]^-{s_2} \\ A \ar@<2pt>[r]^-f & B, \ar@<2pt>[l]^-s}
\]
where the dashed arrows $s_1=\la 1_C,sg\ra$ and $s_2=\la sg,1_C\ra$ are induced by the universal property of the comma objects.
These constructions define respectively the \emph{vertical} and \emph{horizontal comma object change-of-base functors} $V_g\colon \Pt_B(\C)\to \Pt_C(\C)$ and $H_g\colon \Pt_B(\C)\to \Pt_C(\C)$: if $(f,s)$ is a point over $B$, then $V_g(f,s)=(\pi_1,s_1)$ and $H_g(f,s)=(\rho_2,s_2)$, and, for any morphism $\gamma\colon (f,s)\to(f',s')$ in $\Pt_B(\C)$, $V_g(\gamma)$ and $H_g(\gamma)$ are induced by the universal property of the corresponding comma objects.

%

\begin{definition}\label{Ord-enriched protomodularity} A finitely complete $\Ord$-category $\C$ with comma objects and $2$-pullbacks is said to be \defn{lax protomodular} (respectively \defn{colax protomodular}) when the comma object functor $V_{g}$ (respectively $H_g$) is conservative for any morphism $g$ in $\C$.
\end{definition}

As in the classical case, it is easy to prove that, in a pointed $\Ord$-category $\C$ with comma objects, the fact that all functors $V_{g}$ are conservative is equivalent to the fact that the functors $V_{i_A}$ are conservative, where $i_A \colon 0 \to A$ is the only arrow from the zero object (and the same holds for the $H$'s). It follows that, in the pointed context, $\C$ is (co)lax protomodular if and only if the following \emph{(co)lax version of the Split Short Five Lemma} holds:

\begin{theorem}{\emph{(\cite{CMR})}}\label{lax SS5L} A pointed finitely complete $\Ord$-category $\C$ with comma objects and 2-pullbacks is lax protomodular if and only if, given a commutative diagram of split sequences of the form
\[
 \xymatrix{ \comma{i_B}{f} \ar[d]_{a} \ar[r] & A \ar[d]_{b} \ar@<2pt>[r]^-f & B \ar[d]^{c} \ar@<2pt>[l]^-s \\
\comma{i_{B'}}{f'} \ar[r] & A' \ar@<2pt>[r]^-{f'} & B', \ar@<2pt>[l]^-{s'} }
\]
if $a$ and $c$ are isomorphisms, so is $b$.
\end{theorem}

\begin{corollary}{\emph{(\cite{CMR})}}\label{colax SS5L} A pointed finitely complete $\Ord$-category $\C$ with comma objects and 2-pullbacks is colax protomodular if and only if, given a commutative diagram of split sequences of the form
\[ \xymatrix{ \comma{f}{i_B} \ar[d]_{a} \ar[r] & A \ar[d]_{b} \ar@<2pt>[r]^-f & B \ar[d]^{c} \ar@<2pt>[l]^-s \\
\comma{f'}{i_{B'}} \ar[r] & A' \ar@<2pt>[r]^-{f'} & B', \ar@<2pt>[l]^-{s'} } \]
if $a$ and $c$ are isomorphisms, then so is $b$.
\end{corollary}

Our aim here is to characterise pointed varieties of ordered algebras which are (co)lax protomodular using the (co)lax version of the Split Short Five Lemma stated in Theorem~\ref{lax SS5L} and Corollary~\ref{colax SS5L}. The proof is very similar to the one which characterises protomodular varieties of universal algebras~\cite{BJ}.

We start by describing \defn{comma objects} in these varieties. The comma object of an ordered pair of morphisms $(f\colon A\to B, g\colon C\to B)$ in $\Alg{\T}$ is defined as in $\Ord$
\begin{equation}
\label{comma f/g}
\vcenter{\xymatrix@=30pt{ \comma{f}{g} \ar[r]^-{\pi_2} \ar[d]_-{\pi_1} \ar@{}[dr]|-{\le} & C \ar[d]^-g \\
								 A \ar[r]_-f & B,}}
\end{equation}
where $\comma{f}{g}=\{(a,c)\in A\times C: f(a)\le g(c)\}$ with the order induced from that of $A\times C$, and $\pi_1$, $\pi_2$ are projections. We must prove that $\comma{f}{g}$ is a $\T$-subalgebra of $A\times C$: if $\theta\colon X\pitchfork T \to T$ is an operation of arity $X$ in $\T$, then, for any monotone map $\xx\colon X\to f/g$, $\theta_{\comma{f\;}{\;g}}(\xx)=\theta_{A\times C}(\langle\pi_1 \xx,\pi_2 \xx\rangle)=(\theta_A(\pi_1 \xx),\theta_C(\pi_2 \xx))\in \comma{f}{g}$, since $f(\theta_A(\pi_1 \xx))=\theta_B(f \pi_1 \xx)\leq\theta_B(g \pi_2 \xx)=g(\theta_C(\pi_2\xx))$; hence $\theta_{A\times C}$ (co)restricts to $\theta_{\comma{f\;}{\;g}}\colon \comma{f}{g}^{X}\to \comma{f}{g}$.

\begin{theorem}\label{lax proto ord-varieties} Let $\V(\cong \Alg{\T})$ be a pointed variety of ordered algebras. The following statements are equivalent:
\begin{itemize}
	\item[(i)] $\V$ is a lax protomodular category.
	\item[(ii)] there exist $n$ binary operations $\alpha_1,\cdots, \alpha_n$ and an $(n+1)$-ary operation $\theta$ in the theory $\T$ of $\V$ such that:
	\begin{equation}\label{inequalities for lax proto}
	\left\{\begin{array}{l}
		0\le \alpha_i(x,x), \;\mbox{ for all }i\in\{1, \cdots, n\} \\
		x=\theta(\alpha_1(x,y),\cdots, \alpha_n(x,y),y).
		\end{array}\right.
	\end{equation}
\end{itemize}
\end{theorem}
\begin{proof}
(i) $\To$ (ii) Consider the arities $(X,=)$ = $(\{x,y\}, =)$ and $(Z,=)$ = $(\{z\},=)$. We denote by $\pi_x,\pi_y$ the projections of $X\pitchfork T=T^2$ onto $T$ (see Section~\ref{Vars of ordered algs}). We define the (monotone) maps $f\colon (X,=)\to (Z,=)$ by $f(x)=f(y)=z$, and $s\colon (Z,=) \to (X,=)$ by $s(z)=y$. Applying the free $\T$-algebra functor $F\colon \Ord\to\Alg{\T}$, we have $F(X)=\T(T^2,-), F(Z)=\T(T,-)$ and $F(f), F(s)$ are natural transformations whose $T$-components are
\[
	\begin{array}{rcll}
		F(f)_T: \T(T^2,T) & \longrightarrow & \T(T,T) \\
		\xymatrix{T^2 \ar[r]^-\alpha & T} & \longmapsto & \xymatrix@C=60pt{T \ar[r]^{\tau(f)=\la 1_T,1_T\ra} & T^2 \ar[r]^-\alpha & T}
	\end{array}
\]
and
\[
	\begin{array}{rcll}
		F(s)_T: \T(T,T) & \longrightarrow & \T(T^2,T). \\
		\xymatrix{T \ar[r]^-\mu & T} & \longmapsto & \xymatrix@C=40pt{T^2 \ar[r]^{\tau(s)=\pi_y} & T \ar[r]^-\mu & T}
	\end{array}
\]
We adopt the usual simplification and write these $T$-components as $F(f)$ and $F(s)$. Since $\V$ is a pointed variety, the (free) algebras $\T(T^2,T)$ and $\T(T,T)$ admit a \emph{zero}, denoted $0\colon T^2\to T$ and $0\colon T\to T$.

We consider the split sequence
\[
	\xymatrix{K=\comma{i_*}{F(f)} \ar@{^{(}->}[r]^-k & \T(T^2,T) \ar@<2pt>[r]^-{F(f)} & \T(T,T) \ar@<2pt>[l]^-{F(s)}},
\]
where $i_*=i_{\T(T,T)}$ is the unique morphism from $\T(T^0,T)=\{0\}$ to $\T(T,T)$, $K=\{\alpha\in \T(T^2,T): \xymatrix@C=30pt{T \ar@/^/[r]^0 \ar@/_/[r]_-{\alpha\la 1_T,1_T\ra} \ar@{}[r]|{\rotatebox{-90}{$\le$}} & T}\}\subseteq \T(T^2,T)$, and $k$ is the inclusion homomorphism.

Take the (epimorphism, strong monomorphism) factorisation of $\twosum{k}{F(s)}=me$
\begin{equation} \label{me}
\vcenter{
\xymatrix@C=50pt@R=35pt{K\ar[rd]_{e\,\iota_1}\ar@/_1pc/[rdd]_{k}\ar[r]^-{\iota_1} &K+\T(T,T)\ar@{>>}[d]^e & \T(T,T) \ar[l]_-{\iota_2}\ar@<+1.5pt>[ld]^{e\,\iota_2}\ar@/^1pc/@<1.5pt>[ldd]^(.6){F(s)}\\
& M\ar@{ >->}[d]^m\ar@<1.5pt>[ru]^{F(f)m}\\
& \T(T^2,T).\ar@/_1pc/@<1.5pt>[ruu]^(.4){F(f)}}
}
\end{equation}
Note that in the diagram
\[
	\xymatrix{K \ar@{^{(}->}[r]^-{e\,\iota_1} \ar@{=}[d] \pullback & M \ar@{ >->}[d]^-m \\
	K=\comma{i_*}{F(f)} \ar@{^{(}->}[r]^-k \ar[d] \ar@{}[dr]|-{\le} & \T(T^2,T) \ar[d]^-{F(f)} \\
	\T(T^0,T)=\{0\} \ar[r]_-{i_*} & \T(T,T)
	}
\]
the top square is a pullback and the bottom square is a comma object, so that the total diagram is a comma object, i.e. $K=\comma{i_*}{F(f)m}$ (see~\cite{Vassilis,CMR}).

We get a commutative diagram of split sequences
\begin{equation}\label{special SS5L diagram}
\vcenter{
	\xymatrix{ K(=\comma{i_*}{F(f)m}) \ar@{=}[d] \ar@{^{(}->}[r]^-{e\,\iota_1} & M \ar@{ >->}[d]_{m} \ar@<2pt>[r]^-{F(f)m} & \T(T,T) \ar@{=}[d] \ar@<2pt>[l]^-{e\,\iota_2} \\
K=\comma{i_*}{F(f)} \ar@{^{(}->}[r]^-k & \T(T^2,T) \ar@<2pt>[r]^-{F(f)} & \T(T,T). \ar@<2pt>[l]^-{F(s)} } }
\end{equation}

By assumption $\V$ satisfies Theorem~\ref{lax SS5L} applied to diagram \eqref{special SS5L diagram}, thus $m$ is an isomorphism. As a consequence of the surjectivity of $e$, given $\pi_x\in\T(T^2,T)$, there must be a word $w$ in $K+\T(T,T)$ such that $e(w)=\pi_x$. Assume $w=\underline{\beta_1}\overline{\mu_1}\cdots \underline{\beta_n}\overline{\mu_n}$, for some $n\in\N$, $\beta_1, \cdots, \beta_n\in K$, $\mu_1,\cdots, \mu_n\in \T(T,T)$. From the construction of diagram \eqref{me}, we deduce that there exists a $(2n)$-ary operation $\gamma$ such that $\gamma(\beta_1,F(s)(\mu_1),\cdots, \beta_n,F(s)(\mu_n))=\pi_x$ (using here the fact that $k$ is just an inclusion). It is well-known that we can consider another $(2n)$-ary operation $\gamma'$ such that $\gamma'(\beta_1,\cdots, \beta_n,F(s)(\mu_1), \cdots ,F(s)(\mu_n))$ $=$ $\gamma(\beta_1,F(s)(\mu_1),\cdots, \beta_n,F(s)(\mu_n))$ (just by reordering the arguments and using the fact that the product projections are operations). We can now consider an $(n+1)$-ary operation $\gamma''$ such that
\[
	\gamma''(\beta_1,\cdots, \beta_n,F(s)(\mu))=\gamma'(\beta_1,\cdots, \beta_n,F(s)(\mu_1), \cdots ,F(s)(\mu_n)),
\]
where $\mu=\mu_1\cdots \mu_n\in \T(T,T)$. Finally, there exist an $(n+1)$-ary operation $\theta$ and binary operations $\alpha_1=\beta_1$, $\cdots$, $\alpha_{n-1}=\beta_{n-1}$, $\alpha_n=\mu\beta_n$ such that $\theta(\alpha_1, \cdots, \alpha_n, F(s)(1_T))=\pi_x$, i.e. $\theta(\alpha_1, \cdots, \alpha_n, \pi_y)=\pi_x$.

Each $\alpha_i\in K$ gives $0\le \alpha_i \la 1_T,1_T\ra$ in $\T$. Also $\theta \la \alpha_1, \cdots, \alpha_n,\pi_y\ra = \pi_x$ in $\T$. For any $\T$-algebra $A\colon \T \to \Ord$, we get
\[
\begin{array}{rcl}
	A: \T & \longrightarrow & \Ord. \\
	\xymatrix@C=15pt@R=10pt{T \ar[rr]^0 \ar[dr]_-{\la 1_T,1_T\ra} & {} \ar@{}[d]|-{\rotatebox{-90}{$\le$}} & T \\ & T^2 \ar[ur]_-{\alpha_i}} & \longmapsto &
	\xymatrix@C=25pt@R=10pt{A \ar[rr]^0 \ar[dr]_-{\la 1_A,1_A\ra} & {} \ar@{}[d]|-{\rotatebox{-90}{$\le$}} & A \\ & A^2 \ar[ur]_-{\alpha_i}}\\ \\
	\xymatrix@C=60pt{T^2 \ar[r]^-{\la \alpha_1,\cdots,\alpha_n,\pi_y\ra}\ar[dr]_-{\pi_x}^-{\circlearrowright} & T^{n+1} \ar[d]^-{\theta} \\ & T} & \longmapsto &
	\xymatrix@C=60pt{A^2 \ar[r]^-{\la \alpha_1,\cdots,\alpha_n, \pi_y\ra}\ar[dr]_-{\pi_x}^-{\circlearrowright} & A^{n+1} \ar[d]^-{\theta_A} \\ & A}
\end{array}
\]
From the diagrams on the right we conclude that, for elements $x,y\in A$, we have
\[
	\left\{\begin{array}{l}
		 0\le \alpha_i(x,x), \;\mbox{ for all }i\in\{1, \cdots, n\}  \\
		x=\theta_A(\alpha_1(x,y), \cdots, \alpha_n(x,y),y).
	\end{array}\right.
\]

(ii) $\To$ (i) Consider a commutative diagram of split sequences of the form
\[
 \xymatrix{ K=\comma{i_B}{f} \ar[d]|{\cong}_{a=b|_K} \ar@{^{(}->}[r] & A \ar[d]_{b} \ar@<2pt>[r]^-f & B \ar[d]|{\cong}^{c} \ar@<2pt>[l]^-s \\
K'=\comma{i_{B'}}{f'} \ar@{^{(}->}[r] & A' \ar@<2pt>[r]^-{f'} & B', \ar@<2pt>[l]^-{s'} }
\]
where $a$ and $c$ are isomorphisms. We may assume that $B=B'$ and $c=1_B$. By the commutativity of the diagram, we have $a=b|_K$. To prove that $b$ is injective, suppose that $b(x)=b(y)$. For each $i$, $0\le \alpha_i(b(x),b(y))=b(\alpha_i(x,y))$. We conclude that each $\alpha_i(x,y)\in K$, since $f(\alpha_i(x,y))=f'(b(\alpha_i(x,y))\ge 0$; also each $\alpha_i(y,y)\in K$, since $\alpha_i(y,y)\ge 0$ implies $f(\alpha_i(y,y))\ge 0$. From $a(\alpha_i(x,y))=b(\alpha_i(x,y))=\alpha_i(b(x),b(y))=\alpha_i(b(y),b(y))=b(\alpha_i(y,y))=a(\alpha_i(y,y))$ and the fact that $a$ is an isomorphism, we conclude that $\alpha_i(x,y)=\alpha_i(y,y)$, for each $i$. To finish, we have $x=\theta_A(\alpha_1(x,y), \cdots, \alpha_n(x,y),y)=\theta_A(\alpha_1(y,y), \cdots, \alpha_n(y,y),y)=y$.

To prove that $b$ is surjective, consider an arbitrary element $x'\in A'$. Since $f'(x')\in B'=B$ and $f$ is surjective, there exists $x\in A$ such that $f'(x')=f(x)=f'(b(x))$. For each $i$, $\alpha_i(x',b(x))\in K'=\comma{i_{B'}}{f'}$, since $f'(\alpha_i(x',b(x)))=\alpha_i(f'(x'),f'(b(x)))\ge 0$. Then, $z_i:=a^{-1}(\alpha_i(x',b(x)))\in K=\comma{i_B}{f}$. Note that, for each $i$, $b(z_i)=a(z_i)=\alpha_i(x',b(x))$. Hence $\overline{x}=\theta(z_1, \cdots, z_n,x)\in A$ is such that
\[
\begin{array}{ll}
	b(\overline{x}) &=b(\theta(z_1, \cdots, z_n,x)) = \theta_{A'}(b(z_1),\cdots,b(z_n),b(x)) = \\
	                &=\theta_{A'}(\alpha_1(x',b(x)), \cdots, \alpha_n(x',b(x)), b(x)) = x'.
\end{array}
\]
\end{proof}

\begin{corollary}\label{colax proto ord-varieties} Let $\V(\cong \Alg{\T})$ be a pointed variety of ordered algebras. The following statements are equivalent:
\begin{itemize}
	\item[(i)] $\V$ is a colax protomodular category.
	\item[(ii)] there exist $n$ binary operations $\alpha_1,\cdots, \alpha_n$ and an $(n+1)$-ary operation $\theta$ in the theory $\T$ of $\V$ such that:
	\[
	\left\{\begin{array}{l}
		\alpha_i(x,x)\le 0, \;\mbox{ for all } i\in\{1, \cdots, n\} \\
		x=\theta(\alpha_1(x,y),\cdots, \alpha_n(x,y),y).
		\end{array}\right.
	\]
\end{itemize}
\end{corollary}
\begin{proof}
The proof is similar to that of Theorem~\ref{lax proto ord-varieties}. The main difference is that now we use the characterisation of colax protomodularity from Corollary~\ref{colax SS5L}. The fact that, for each $i$, the existing $\alpha_i$ belongs to $K=\comma{f}{i_B}$ implies that each $\alpha_i(y,y)=\alpha_i(f(x),f(y))=f(\alpha_i(x,y))\le 0$.
\end{proof}\

\begin{example}\label{ex of (co)lax proto var}
Varieties of ordered algebras provide examples showing that the following question, posed in~\cite{CMR}, has a positive answer:
\begin{center} \textsf{Is there any lax protomodular or colax protomodular non-degenerate $\Ord$-category?} \end{center}
(where by degenerate $\Ord$-enrichment we mean one whose hom-sets preorders are symmetric).
\begin{itemize}
	\item[1.] Let $\V$ be the pointed variety of ordered algebras with two binary operations $\alpha$ and $\theta$ as in Theorem~\ref{lax proto ord-varieties}(ii) when $n=1$, so that it is a lax protomodular category.  Any ordered set $A$ with a bottom element $\bot$ gives an example of an algebra of $\V$, with $\bot=0$, and $\alpha_A(x,y)=\theta_A(x,y)=x$. In particular, $\N$ with its usual order is an example of such an algebra.
	\item[2.] Similarly, let $\V$ be a pointed variety of ordered algebras which is a colax protomodular category for which $n=1$ and there exist binary operations $\alpha$ and $\theta$ as in Corollary~\ref{colax proto ord-varieties}(ii). Any ordered set $A$ with a top element $\top$ gives an example of an algebra of $\V$, with $\top=0$, and $\alpha_A(x,y)=\theta_A(x,y)=x$. In particular, the set of non-positive integers with its usual order is an example of such an algebra.
\end{itemize}
\end{example}

\section{The Mal'tsev property of varieties of ordered algebras}\label{Section: Mal'tsev}
We recall from~\cite{CR} the main topics concerning the notion of $\Ord$-Mal'tsev category.

Let $\C$ denote an $\Ord$-category.
A morphism $m\colon U\to V$ is said to be \defn{full} when: for any morphisms $a,a'\colon A\to U$ such that $ma \moi ma'$, $a\moi a'$; equivalently, $ma\moi ma'$ if and only if $a\moi a'$. In the $\Ord$-enriched context, all morphisms are faithful. The fully faithful (mono)morphisms are called \defn{ff-(mono)morphisms} (ff'' stands for ``fully faithful''); see~\cite{KV, Vassilis}. We denote ff-monomorphisms with arrows of the type $\mono$. A morphism $e\colon A\to B$ is called \defn{surjective on objects}, or \defn{so-morphism}, when $e$ is left orthogonal to every ff-monomorphism $m$: the usual diagonal fill-in property holds
\[
	\xymatrix{A \ar[r]^-e \ar[d]_-u & B \ar[d]^-v \ar@{.>}[dl]|-d \\ U \ar@{ >->}[r]_-m & V.}
\]
When $\C$ has binary 2-products, then every so-morphism is necessarily an epimorphism.

A \defn{relation} from an object $A$ to an object $B$ of $\C$ is a span $A\stackrel{r_1}{\longleftarrow} R \stackrel{r_2}{\longrightarrow} B$ such that $(r_1,r_2)$ is jointly ff-monomorphic. The \emph{opposite} relation $R^\circ$ is the span $B\stackrel{r_2}{\longleftarrow} R \stackrel{r_1}{\longrightarrow} A$. When $\C$ admits binary 2-products, a relation is given by an ff-monomorphism $\la r_1,r_2\ra \colon R\mono A\times B$. When $A=B$, we say that $R$ is a relation on $A$.

Given morphisms $x\colon W\to A$ and $y\colon W\to B$ of $\C$, we use the notation $(x,y)\in_W R$, or $xRy$, when there exists a (unique) factorisation $\la x,y\ra^R$ making the following diagram commutative
\[
	\xymatrix@R=10pt{ & A \\
		W \ar[ur]^-{x} \ar[dr]_-{y} \ar@{.>}[rr]^{\langle x,y\rangle^R} & & R. \ar[ul]_-{r_1} \ar[dl]^-{r_2}\\
		& B}
\]
A relation $A\stackrel{r_1}{\longleftarrow} R \stackrel{r_2}{\longrightarrow} B$ in $\C$ is called an \defn{ideal} when, given generalised elements $x,x'\colon W\to A$, $y,y'\colon W\to B$, we have
\begin{equation}\label{w-c def}
	\left( x'\le x \;\wedge\; (x,y)\in_W R \;\wedge\; y\le y' \right) \Rightarrow (x',y')\in_W R.
\end{equation}
\noindent Note that an ideal $A\stackrel{r_1}{\longleftarrow} R \stackrel{r_2}{\longrightarrow} B$ is a relation, by definition, thus
$(r_1,r_2)$ is jointly ff-monomorphic. We use the notation $R\colon A\looparrowright B$ for ideals. An ideal $R\colon A\looparrowright A$ which is reflexive and transitive is called a \defn{congruence}.

Any comma object as in diagram~\eqref{comma f/g} gives an example of a relation $A\stackrel{\pi_1}{\longleftarrow} \comma{f}{g} \stackrel{\pi_2}{\longrightarrow} C$ in $\C$; this relation is also an ideal, and any comma object of the type $\comma{f}{f}$ is a \emph{kernel} congruence.

\begin{definition}\label{Ord-Mal'tsev} An $\Ord$-category $\C$ is called an \defn{$\Ord$-Mal'tsev category} when every ideal $D\colon A\looparrowright B$ satisfies the property: given morphisms $x, u, u'\colon W\to A$, $y,y',v\colon W\to B$ such that $(x,y)\in_W D$, $y\le y'$, $(u,y')\in_W D$, $u\le u'$ and $(u',v)\in_W D$, then $(x,v)\in_W D$. This property may be pictured as
\begin{equation}\label{picture Ord-difunctional}
	\begin{array}{ccc}
	 x & D & y \vspace{-10pt}\\
	   &   & \rotatebox{-90}{$\le$} \\
	 u & D & y' \vspace{-10pt}\\
	\rotatebox{-90}{$\le$} \\
	 u' & D & v \\
	\hline
	x & D & v.
	\end{array}
\end{equation}
\end{definition}

There are several different approaches to regularity for enriched categories (see for instance~\cite{BourkeGarner} for a general framework of such notions). We chose to use an approach which allows for an easy calculus of (order) ideals. An $\Ord$-category $\C$ is called \defn{regular} when $\C$ has finite (weighted) limits, $\C$ admits an (so-morphism,ff-monomorphism) factorisation system, so-morphisms are stable under $2$-pullbacks in $\C$ and every so-morphism is a bicoinserter (of its comma object). When $\C$ is a regular $\Ord$-category, we may compose relations (similarly as in the 1-dimensional regular categorical context); see Definition 1.2 in~\cite{CR}. Any variety of ordered algebras is a regular $\Ord$-category.

\begin{theorem}{\emph{(\cite{CR})}}\label{Ord-enriched for CKP} Let $\C$ be a regular $\Ord$-category. Then the following statements are equivalent:
\begin{itemize}
	\item[(i)] $\C$ is a (regular) $\Ord$-Mal'tsev category.
	\item[(ii)] $RS\cong SR$, for any congruences $R,S\colon A\looparrowright A$ on an object $A$.
\end{itemize}
\end{theorem}

\begin{theorem}\label{Ord-Maltsev ord-varieties} Let $\V(\cong \Alg{\T})$ be a variety of ordered algebras. The following statements are equivalent:
\begin{itemize}
	\item[(i)] $\V$ is an $\Ord$-Mal'tsev category.
	\item[(ii)] There exists a ternary operation $\rho$ in the theory $\T$ of $\V$ such that:
	\[
	\left\{\begin{array}{l}
		a\le \rho(a,b,c), \;\text{when}\;\; b\le c, \\
		\rho(u,v,w)\le w, \;\text{when}\;\; u\le v.
		\end{array}\right.
	\]
\end{itemize}
\end{theorem}
\begin{proof} (i) $\To$ (ii) Consider the arities $(X,=)$ = $(\{x_1,x_2,x_3\}, =)$ and $(Z,\le)$ = $(\{z_1,z_2,z_3,z_4\},z_2\le z_3)$. We denote by $\pi_i$ the projections $\pi_{x_i}$ of $X\pitchfork T=T^3$ onto $T$ ($i=1,2,3$), and by $\overline{\pi}_i$ the operations $\pi_{z_i}\in\T(Z\pitchfork T,T)$ ($i=1,2,3,4$). The order on $Z$ induces an order on the $\overline{{\pi}_i}'s$: $\overline{\pi}_2\leq \overline{\pi}_3$. Consider now the (monotone) maps $f,g\colon (X,=)\to (Z,\le)$ defined by $f(x_i)=z_i$, $g(x_i)=z_{i+1}$, $i\in\{1,2,3\}$. Applying the free $\T$-algebra functor $F\colon \Ord\to\Alg{\T}$,
we have $F(X)=\T(T^3,-)$, $F(Z)=\T(Z\pitchfork T,-)$ and $F(f)$ is a natural transformation whose $T$-component is
\[
	\begin{array}{rcll}
		F(f)_T: \T(T^3,T) & \longrightarrow & \T(Z\pitchfork T,T). \\
		\xymatrix{T^3 \ar[r]^-{\alpha} & T} & \longmapsto & \xymatrix@C=70pt{Z\pitchfork T \ar[r]^-{\tau(f)=\la \overline{\pi}_1,\overline{\pi}_2,\overline{\pi_3}\ra} & T^3 \ar[r]^-\alpha & T}
	\end{array}
\]
In particular, $F(f)_T(\pi_i)=\overline{\pi}_i$; analogously, $F(g)_T(\alpha)=\alpha\tau(g)$, with $\tau(g)=\la \overline{\pi}_2,\overline{\pi}_3,\overline{\pi}_4\ra$, thus $F(g)_T(\pi_i)=\overline{\pi}_{i+1}$, for each $i$.

Consider the congruences on $\T(T^3,T)$ given by the comma objects $R=\comma{F(f)}{F(f)}$ and $S=\comma{F(g)}{F(g)}$. Then
\[
	\begin{array}{l}
	  (\pi_1,\pi_2)\in S, \text{ since}\;\, F(g)(\pi_1)=\overline{\pi}_2 \le \overline{\pi}_3=F(g)(\pi_2)\;\text{, and}\\
	  (\pi_2,\pi_3)\in R, \text{ since}\;\, F(f)(\pi_2)=\overline{\pi}_2 \le \overline{\pi}_3=F(f)(\pi_3),
	\end{array}
\]
i.e. $\pi_1 RS \pi_3$. Since $\V$ is a regular $\Ord$-Mal'tsev category, we have $RS\cong SR$, by Theorem~\ref{Ord-enriched for CKP}; thus, $\pi_1 SR \pi_3$. Consequently, there exists a ternary operation $\rho\in \T(T^3,T)$ such that
\begin{equation}\label{eq:p}
	\begin{array}{lcl}
		\pi_1 R \rho  & : & F(f)(\pi_1)\le F(f)(\rho) \Leftrightarrow \overline{\pi}_1\leq \rho\tau(f) \vspace{5pt}\\
		& & \xymatrix@R=5pt@C=40pt{Z\pitchfork T \ar[dr]_(.4){\tau(f)=\la \overline{\pi}_1,\overline{\pi}_2,\overline{\pi}_3\ra\;\;\;\;\;\;} \ar[rr]^-{\overline{\pi}_1}_-{\rotatebox{-90}{$\le$}} & & T \\ & T^3 \ar[ur]_-\rho} \vspace{5pt} \\
		\rho S \pi_3  & : & F(g)(\rho)\le F(g)(\pi_3) \Leftrightarrow \rho\tau(g)\leq \overline{\pi}_4 \vspace{5pt}\\
		& & \xymatrix@R=5pt@C=40pt{ & T^3 \ar[dr]^-p \\ Z\pitchfork T \ar[ur]^(0.4){\tau(g)=\la \overline{\pi}_2,\overline{\pi}_3,\overline{\pi}_4\ra\;\;\;\;\;\;} \ar[rr]_-{\overline{\pi}_4}^-{\rotatebox{-90}{$\le$}} & & T.} \\
	\end{array}
\end{equation}

For any $\T$-algebra $A\colon \T \to \Ord$, the image under $A$ of the first diagram of \eqref{eq:p} is
\[\xymatrix@R=5pt{A^Z \ar[rr]^-{A(\overline{\pi}_1)}_-{\rotatebox{-90}{$\le$}} \ar[dr]_(.4){A(\tau(f)=A(\la \overline{\pi}_1,\overline{\pi}_2,\overline{\pi}_3\ra)\;\;\;\;\;\;\;\;} & &  A.\\ & A^3 \ar[ur]_(.6){\rho_A} }\]

Consider elements $a,b,c\in A$ with $b\le c$ and the monotone map $\xx\colon Z\to A$ defined by $\xx(z_1)=a$, $\xx(z_2)=b$ and $\xx(z_3)=\xx(z_4)=c$. Then $A(\overline{\pi}_1)(\xx)=\xx(z_1)=a$ and $A(\tau(f))(\xx)=(\xx(z_1),\xx(z_2),\xx(z_3))=(a,b,c)$. So, the order between the morphisms of the triangle above translates to
\begin{equation}\label{M1}
	a\le \rho_A(a,b,c),\; \text{when}\;\; b\le c.
\end{equation}

Similarly, applying the functor $A$ to the second diagram of \eqref{eq:p} we obtain
\[\xymatrix@R=5pt{ & A^3 \ar[dr]^(.6){\rho_A} \\
		A^Z \ar[rr]_-{A(\overline{\pi}_4)}^-{\rotatebox{-90}{$\le$}} \ar[ur]^(.4){A(\tau(g))=A(\la \overline{\pi}_2,\overline{\pi}_3,\overline{\pi}_4\ra)\;\;\;\;\;\;\;\;} & & A. }\]
Consider elements $u,v,w\in A$ with $u\le v$ and the monotone map $\yy\colon Z\to A$ defined by $\yy(z_1)=\yy(z_2)=u$, $\yy(z_3)=v$ and $\yy(z_4)=w$. Then $A(\overline{\pi}_4)(\yy)=\yy(z_4)=w$ and $A(\tau(g))(\yy)=(\yy(z_2), \yy(z_3), \yy(z_4))=(u,v,w)$. So, the order between the morphisms of the triangle above translates to
\begin{equation}\label{M2}
	\rho_A(u,v,w)\le w,\; \text{when}\;\; u\le v.
\end{equation}

(ii) $\To$ (i) Suppose that $R$ and $S$ are congruences on some $\T$-algebra $A$ such that $a R b S c$. We have
\[
	\begin{array}{ccccc}
		a & S & a & R & b \\
		b & S & b & R & b \\
		b & S & c & R & c.
	\end{array}
\]
Applying $\rho_A$ to the vertical triples, we get $\rho_A(a,b,b)\; S\; \rho_A(a,b,c)\; R\; \rho_A(b,b,c)$. By assumption, $a\le \rho_A(a,b,b)$ and $\rho_A(b,b,c)\le c$. Using the fact that $R$ and $S$ are ideals, we conclude that $a\; S\; \rho_A(a,b,c)\; R\; c$, proving that $SR\subseteq RS$. The proof that $RS\subseteq SR$ is similar.
\end{proof}

\begin{corollary}\label{Ord-Maltsev vars are degenerate} Let $\V\cong \Alg{\T}$ be a variety of ordered algebras which is an $\Ord$-Mal'tsev category. Then any $\T$-algebra has degenerate order.
\end{corollary}
\begin{proof}
	Consider a $\T$-algebra $A$ and suppose that $x\le y$ in $A$. Using the inequalities of Theorem~\ref{Ord-Maltsev ord-varieties}(ii) and the fact that $\rho_A$ is monotone, we get
\[
	y\le \rho_A(y,x,x)\leq \rho_A(y,y,x)\le x.
\]
\end{proof}

It is well-known that any protomodular category is a Mal'tsev category~\cite{MCFPO}. However, this implication is false in our order enriched context. Moreover, a variety of ordered algebras which is a $\Pos$-Mal'tsev category (i.e. a $\Pos$-enriched category satisfying Definition~\ref{Ord-Mal'tsev}) is necessarily a Mal'tsev category, since the order of the algebras is discrete and condition (ii) of Theorem~\ref{Ord-Maltsev ord-varieties} coincides with \eqref{proto ops}.

\begin{proposition}\label{Ord-proto does not imply Ord-Maltsev} A variety of ordered algebras which is a (co)lax protomodular category is not necessarily an $\Ord$-Mal'tsev category.
\end{proposition}
\begin{proof} This follows immediately from Example~\ref{ex of (co)lax proto var} and Corollary~\ref{Ord-Maltsev vars are degenerate}.
\end{proof}

\section{\Ord-Mal'tsev non-coherent ordered varieties}\label{Section: non-coherent}
As it is well-known, the category of ordered groups and monotone homomorphisms is not a variety of ordered algebras in the sense we presented in Section 1. Indeed, while the addition (not necessarily commutative) is monotone, the inversion operation is anti-monotone. In this section we will use a more general notion of variety of ordered algebras which includes the category of ordered groups and ordered vector spaces, for instance.

The definition of variety we presented in Section 1 uses algebraic theories as investigated by Power and his co-authors (see \cite{Power, NPower, PPower}; see also \cite{LR}), where all the operations of the algebras are morphisms in $\Ord$, that is are monotone. Following \cite{Adamek-Milius}, we say that these operations are \emph{coherent}, and that the variety is \emph{coherent}. A \emph{(not necessarily coherent) variety} is defined as before, but with no condition on the monotonicity of the operations. We point out that, as shown in \cite{Adamek-Milius} for the case of $\Pos$, while coherent varieties are the categories of algebras for finitary \emph{$\Pos$-enriched} monads on $\Pos$, (not necessarily coherent) varieties are the categories of algebras for finitary (not necessarily $\Pos$-enriched) monads on $\Pos$.

The aim of this Section is to show that for non-coherent varieties of ordered algebras the Mal'tsev condition does not trivialise, i.e. there exist algebras with non-degenerate order.

\begin{example}
Consider the variety $\V\cong \Alg{\T}$ of ordered algebras where the theory $\T$ has a 0-ary operation $0$, a monotone associative binary operation $+\colon T^2\to T$ with neutral element 0, and a (not necessarily monotone) operation $\ominus\colon X\pitchfork T \to T$, with arity $(X,\le)=(\{x,y\}, x\ge y)$, satisfying the following conditions:
\begin{enumerate}
	\item $0\leq a$;
	\item $a\ominus 0=a\;\;\&\;\;a\ominus a=0$.
\end{enumerate}
A natural example of such a $\T$-algebra is the ordered monoid $\N$ of natural numbers. Indeed, before we prove that this category is $\Ord$-Mal'tsev, let us check that it contains, as a subcategory (up to isomorphism), the category $\Ab\Mon_c$ of abelian monoids with cancellation. Given an abelian monoid with cancellation $(A,+,0)$, for $a,b\in A$ we define $a\le b$ if there exists $b_a\in A$ (necessarily unique) such that $a+b_a=b$. It is clear that $0\leq a$ for every $a\in A$, and it is easy to check that $\leq$ is reflexive and transitive (but not necessarily anti-symmetric), and that this order makes $+$ monotone (using commutativity of $+$). We define a new operation $\ominus\colon A^X\to A$ assigning to $(b,a)$, with $b\ge a$, $b\ominus a=b_a$, that is $b\ominus a$ is the unique element of $A$ such that $a+(b\ominus a)=b$. Then $a\ominus 0=a$ and $a\ominus a=0$ as required. We remark that this operation is in general not monotone.

This correspondence is functorial: if $f\colon A\to B$ is a morphism in $\Ab\Mon_c$, that is, $f$ is a homomorphism of monoids, then $f$ preserves all the operations, $0$, $+$, and $\ominus$, so it is a morphism in $\V$. Moreover, this functor is full and faithful, and injective on objects.

Now to prove that $\V$ is $\Ord$-Mal'tsev we follow the strategy used in Example 6.2 of \cite{CR}:
Let $\la d_1,d_2 \ra\colon D\to A\times B$ be an ideal in $\V$, and let $f_1,f_2,f_3\colon W\to A$ and $g_1,g_2,g_3\colon W\to B$ be such that $f_2\le f_3$, $g_1\le g_2$, and $f_i\,D\,g_i$ for $i=1,2,3$. We want to show that $f_1\,D\,g_3$. First note that, by (1), the zero homomorphism verifies $0\le h$ for every morphism $h$. In particular $0\le f_i$, hence $0\,D\,g_i$, because $D$ is an ideal, and $\langle 0,g_i\rangle^D\leq \langle f_i,g_i\rangle^D\colon W\to D$, for each $i$. Hence, for every $w\in W$,
\[\langle f_1,g_1\rangle^D(w)\ominus \la 0,g_1\rangle^D(w)=(f_1\ominus 0(w),g_1(w)\ominus g_1(w))\stackrel{(2)}{=}(f_1(w),0)\]
belongs to $D$; that is $f_1\,D\,0$. Using addition on $f_1\,D\,0$ and $0\,D\,g_3$ gives $f_1\,D\,g_3$ as claimed.

\end{example}

\begin{remark}
As pointed out in \cite[Example 6.4]{CR}, the pointwise enrichment of category $\Ord\Grp$ makes the order trivialise, and so $\Ord\Grp$ would be $\Ord$-Mal'tsev only if it were Mal'tsev, and it was shown in \cite[Theorem 4.6]{CM-FM} that this category is not a Mal'tsev category.

We note, however, that it has a (non-monotone) ternary operation $\rho\colon T^3\to T$, defined classically as $\rho(x,y,z)=xy^{-1}z$, verifying the conditions of Theorem \ref{Ord-Maltsev ord-varieties}. This shows that the characterisation obtained for \Ord-Mal'tsev coherent ordered varieties does not extend to the non-coherent case.
\end{remark}

\section*{Acknowledgments}
We are grateful to Dirk Hofmann for his insight on ordered algebras and interesting exchange of ideas with respect to this work, and to the anonymous referee for helping us to improve the final presentation.


\end{document}